\documentclass[11pt]{article}
\usepackage{geometry,amsmath,amsthm,amssymb,latexsym}

\usepackage[svgnames]{xcolor}
\usepackage[colorlinks=true,linkcolor=MidnightBlue,citecolor=MidnightBlue,urlcolor=MidnightBlue,pdfpagelabels=false]{hyperref}

\usepackage{thmtools}
\theoremstyle{plain}
  \declaretheorem[numberwithin=section]{theorem}
  \declaretheorem[numberlike=theorem]{corollary}
  \declaretheorem[numberlike=theorem]{proposition}
  \declaretheorem[numberlike=theorem]{lemma}

\theoremstyle{definition}
  
  \declaretheorem[numberlike=theorem]{example}

\begin{document}

\title{Criteria for the integrality of $n$th roots of power series}

\author{
  John Pomerat\quad\&\quad Armin Straub
  \\ University of South Alabama
}

\date{June 17, 2024}

\maketitle

\begin{abstract}
  Heninger, Rains and Sloane raised the question of which power series with
  integer coefficients can be written as the $n$th power of another power
  series with integer coefficients and constant term $1$. We provide necessary
  and sufficient conditions, as well as compare with a general integrality
  criterion due to Dieudonn\'e and Dwork that can be applied to this
  question as well.
\end{abstract}

\section{Introduction}

Heninger, Rains and Sloane \cite{hrs-integrality} raise and investigate the
question of charactizing when a formal power series $f (x) \in 1 + x\mathbb{Z}
[[x]]$ has an $n$th root $g (x) \in 1 + x\mathbb{Z} [[x]]$. That is, given $f
(x)$, does there exist $g (x) \in 1 + x\mathbb{Z} [[x]]$ such that $f (x) = g
(x)^n$? Among other results, a necessary condition is proved in
\cite[Theorem~16]{hrs-integrality} in the case where $n$ is a prime. In the
following, we extend this condition to the case of prime powers and prove that
it is not only necessary but also sufficient. As explained below, the case of
general $n$ can be reduced to the case of prime powers so that
Theorem~\ref{thm:ps:int} can be considered an answer to the question of
Heninger, Rains and Sloane.

\begin{theorem}
  \label{thm:ps:int}Let $a_1, a_2, \ldots \in \mathbb{Z}_p$ and $r \in
  \mathbb{Z}_{> 0}$. The power series $(1 + a_1 x + a_2 x^2 + \ldots)^{1 /
  p^r}$ has $p$-integral coefficients if and only if
  \begin{equation}
    1 + a_1 x + a_2 x^2 + \ldots \equiv (1 + a_{p^r} x + a_{2 p^r} x^2 +
    \ldots)^{p^r} \pmod{p^{r + 1}} . \label{eq:ps:int}
  \end{equation}
\end{theorem}

Here, as well as throughout, $p$ denotes a prime. We write $\mathbb{Z}_p$ for
the ring of $p$-adic integers as well as $\mathbb{Q}_p$ for its fraction field
(for nice introductions to $p$-adic numbers, we refer to \cite{koblitz-p}
and \cite{robert-p}). For $A, B \in \mathbb{Q}_p [[x]]$, we will write
either $A \equiv B \pmod{p^r \mathbb{Z}_p [[x]]}$ or, more
succinctly, $A \equiv B \pmod{p^r}$ to mean that $A - B \in
p^r \mathbb{Z}_p [[x]]$. Since $\mathbb{Z}$ and $\mathbb{Q}$ naturally embed
into $\mathbb{Z}_p$ and $\mathbb{Q}_p$, we may without possible confusion
treat them as subsets. For instance, given $a \in \mathbb{Q}$, we say that $a
\in \mathbb{Z}_p$ if and only if the (reduced) denominator of $a$ is not a
multiple of $p$.

As reviewed in Section~\ref{sec:review}, if $f (x) = g (x)^n$ for two power
series $f (x), g (x)$ in $1 + x\mathbb{Q} [[x]]$, then we necessarily have $g
(x) = f (x)^{1 / n}$ where $f (x)^{1 / n}$ is well-defined as a power series
by the binomial expansion \eqref{eq:bin}. A more general (but ultimately
equivalent) version of the question by Heninger, Rains and Sloane therefore is
to ask: under which conditions on $f (x) \in 1 + x\mathbb{Z} [[x]]$ and
$\lambda \in \mathbb{Q}$ has the power series $f (x)^{\lambda}$ integral
coefficients itself? This question for rational exponents $\lambda$ readily
reduces to the question for $\lambda = 1 / p^r$ where $p^r$ is a prime power
(and this is the case covered by Theorem~\ref{thm:ps:int}).

\begin{proposition}
  \label{prop:lambda:pr}Suppose $f (x) \in 1 + x\mathbb{Z} [[x]]$ and $\lambda
  \in \mathbb{Q}$. Write $\lambda = n / (p_1^{r_1} \cdots p_m^{r_m})$, where
  $r_j \in \mathbb{Z}_{> 0}$ and where the $p_j$ are distinct primes not
  dividing $n \in \mathbb{Z}$. Then we have $f (x)^{\lambda} \in \mathbb{Z}
  [[x]]$ if and only if $f (x)^{1 / p^r} \in \mathbb{Z}_p [[x]]$ for all $p^r
  \in \{ p_1^{r_1}, \ldots, p_m^{r_m} \}$.
\end{proposition}

\begin{proof}
  Observe that $f (x)^{\lambda} \in \mathbb{Z} [[x]]$ if and only if $f
  (x)^{\lambda} \in \mathbb{Z}_p [[x]]$ for all primes $p$. If $p$ is a prime
  not dividing the denominator of $\lambda$, then $\lambda \in \mathbb{Z}_p$
  and it follows readily that $f (x)^{\lambda} \in \mathbb{Z}_p [[x]]$ (see
  the discussion after the binomial expansion \eqref{eq:bin}). On the other
  hand, suppose that $p^r \in \{ p_1^{r_1}, \ldots, p_m^{r_m} \}$. Then, as
  reviewed in Section~\ref{sec:review}, $f (x)^{\lambda} \in \mathbb{Z}_p
  [[x]]$ if and only if $f (x)^{\lambda \mu} = f (x)^{1 / p^r} \in
  \mathbb{Z}_p [[x]]$ where $\mu = 1 / (p^r \lambda) \in
  \mathbb{Z}_p^{\times}$.
\end{proof}

Our criterion for the $p$-integrality of power series $f (x)^{1 / p^r}$ proved
in Theorem~\ref{thm:ps:int} is particularly easy to use for certain basic
applications, such as the ones illustrated in the examples below.

\begin{example}
  \label{eg:intro:int}As a first, particularly simple application of
  Theorem~\ref{thm:ps:int}, let us confirm that the power series
  \begin{equation}
    (1 - 12 x + 12 x^2 + 8 x^3)^{- 1 / 6} = 1 + 2 x + 12 x^2 + 92 x^3 + 784
    x^4 + 7056 x^5 + \ldots \label{eq:intro:int}
  \end{equation}
  has integer coefficients. By Proposition~\ref{prop:lambda:pr}, this is the
  case if and only if $f (x)^{1 / p} \in \mathbb{Z}_p [[x]]$ for $p \in \{ 2,
  3 \}$ where $f (x) = 1 - 12 x + 12 x^2 + 8 x^3$. In the case $p = 2$, by
  Theorem~\ref{thm:ps:int}, we have $f (x)^{1 / 2} \in \mathbb{Z}_2 [[x]]$ if
  and only if $f (x) \equiv (1 + 12 x)^2 \pmod{4}$, which
  is clearly true. Likewise, $f (x)^{1 / 3} \in \mathbb{Z}_3 [[x]]$ if and
  only if $f (x) \equiv (1 + 8 x)^3 \pmod{9}$, which is
  true as well. On the other hand, the same analysis shows that, for instance,
  $(1 - 12 x + 12 x^2 - 8 x^3)^{- 1 / 6}$ does not have coefficients that are
  all integers (indeed, the coefficient of $x^3$ is $284 / 3$).
\end{example}

\begin{example}
  \label{eg:pr1}It follows from Theorem~\ref{thm:ps:int} that, if $a_1, a_2,
  \ldots, a_d \in \mathbb{Z}_p$ with $d < p^r$, then $(1 + a_1 x + a_2 x^2 +
  \ldots + a_d x^d)^{1 / p^r} \in \mathbb{Z}_p [[x]]$ if and only if $a_1,
  a_2, \ldots, a_d \in p^{r + 1} \mathbb{Z}_p$. In particular, $(1 + a x)^{1 /
  p^r} \in \mathbb{Z}_p [[x]]$ if and only if $a \in p^{r + 1} \mathbb{Z}_p$.
\end{example}

\begin{example}
  \label{eg:bin:kcentral}The coefficients $c (n, k)$ in the power series
  expansions
  \begin{equation*}
    \sum_{n = 0}^{\infty} c (n, k) x^n = (1 - k^2 x)^{- 1 / k}
  \end{equation*}
  are studied in \cite{ams-centralbinomials}. These numbers are referred to
  as $k$-central binomial coefficients, with $c (n, 2) = \binom{2 n}{n}$ the
  usual central binomial coefficients. It follows from the binomial theorem,
  see \cite[Proposition~2.1]{ams-centralbinomials}, that
  \begin{equation*}
    c (n, k) = (- 1)^n \binom{- 1 / k}{n} k^{2 n} = \frac{k^n}{n!} \prod_{m =
     1}^{n - 1} (1 + k m) .
  \end{equation*}
  In \cite[Theorem~2.2]{ams-centralbinomials}, these coefficients are proved
  to be positive integers. For instance, in the case $k = 4$, this proves that
  $(1 - 16 x)^{- 1 / 4} \in \mathbb{Z} [[x]]$. This, however, is not best
  possible in the sense that, already, $(1 - 8 x)^{- 1 / 4} \in \mathbb{Z}
  [[x]]$ (this observation is equivalent to the fact that $2^n$ divides $c (n,
  4)$). More generally, we find, as an application of
  Theorem~\ref{thm:ps:int}, that, for any nonzero $k \in \mathbb{Z}$,
  \begin{equation}
    (1 - k \operatorname{rad} (k) x)^{- 1 / k} \in \mathbb{Z} [[x]],
    \label{eq:bin:k:radk}
  \end{equation}
  where $\operatorname{rad} (k)$ denotes the largest squarefree integer dividing $k$
  (for instance, $\operatorname{rad} (4) = 2$ and $\operatorname{rad} (24) = 6$). The result
  \eqref{eq:bin:k:radk} is a strengthening of
  \cite[Theorem~2.2]{ams-centralbinomials} and follows directly from the
  observation made in the previous example that $(1 + a x)^{1 / p^r} \in
  \mathbb{Z}_p [[x]]$ if and only if $a \in p^{r + 1} \mathbb{Z}_p$. Moreover,
  it follows from Theorem~\ref{thm:ps:int} combined with
  Proposition~\ref{prop:lambda:pr} that \eqref{eq:bin:k:radk} is strongest
  possible in the sense that, with the exponent $- 1 / k$ fixed, $(1 - a x)^{-
  1 / k} \in \mathbb{Z} [[x]]$ if and only if $a$ is a multiple of $k
  \operatorname{rad} (k)$.
\end{example}

\begin{example}
  \label{eg:nthroot:mod}Suppose that $f (x) \in 1 + x\mathbb{Z} [[x]]$ and $n
  \in \mathbb{Z}_{> 0}$. Heninger, Rains and Sloane proved
  \cite[Theorem~1]{hrs-integrality} that the power series $f (x)^{1 / n}$
  has integer coefficients if and only if this is true for $f (x)$ with its
  coefficients reduced modulo $n \operatorname{rad} (n)$. As in the previous example,
  this result can be deduced as a corollary of Theorem~\ref{thm:ps:int}.
\end{example}

\begin{example}
  For $a, b, c \in \mathbb{Z}$, Noe \cite{noe-trinomial} studies the
  generalized trinomial coefficients
  \begin{equation*}
    T_n (a, b, c) = [x^n] (a + b x + c x^2)^n,
  \end{equation*}
  where $[x^n] f (x)$ denotes the coefficient of $x^n$ in $f (x)$, and notes
  that they have the ordinary generating function
  \begin{equation}
    \sum_{n \geq 0} T_n (a, b, c) x^n = \frac{1}{\sqrt{1 - 2 b x + d
    x^2}} \label{eq:trinomial:ogf}
  \end{equation}
  where $d = b^2 - 4 a c$ (note that $T_n (a, b, c)$ only depends on $a c$,
  not on the individual values of $a$ and $c$). As an application of
  Theorem~\ref{thm:ps:int}, we can conclude that these are the only instances
  in which the power series $(1 + \alpha x + \beta x^2)^{- 1 / 2}$, with
  $\alpha, \beta \in \mathbb{Z}$, has integer coefficients. Indeed, applying
  Theorem~\ref{thm:ps:int}, the condition~\eqref{eq:ps:int} with $p = 2$, $r =
  1$ becomes
  \begin{equation*}
    1 + \alpha x + \beta x^2 \equiv (1 + \beta x)^2 \pmod{4},
  \end{equation*}
  which is equivalent to $\alpha \equiv 2 \beta$ and $\beta \equiv \beta^2
  \pmod{4}$. In other words, $\alpha \equiv 2 \beta$ and
  $\beta \equiv 0, 1 \pmod{4}$. Since $\alpha$ is
  necessarily even, we may set $b = - \alpha / 2 \in \mathbb{Z}$ and $d =
  \beta$ to match the right-hand side of \eqref{eq:trinomial:ogf}. In terms of
  $b$ and $d$, the conditions for integrality reduce to
  \begin{equation*}
    b \equiv d \pmod{2}, \quad d \equiv 0, 1 \pmod{4},
  \end{equation*}
  which can be combined to $d \equiv b^2 \pmod{4}$. In
  particular, letting $a = 1$ and $c = (b^2 - d) / 4 \in \mathbb{Z}$, we
  conclude that, if the coefficients of $(1 + \alpha x + \beta x^2)^{- 1 / 2}$
  are integers, then they are given by the generalized trinomial coefficients
  $T_n (a, b, c)$.
\end{example}

\begin{example}
  \label{eg:quadr:int}As another application of Theorem~\ref{thm:ps:int}, let
  us characterize the values $a, b \in \mathbb{Z}$ and $\lambda \in
  \mathbb{Q}$, for which the power series $(1 + a x + b x^2)^{\lambda}$ has
  integer coefficients. As in Example~\ref{eg:pr1}, it follows from
  Theorem~\ref{thm:ps:int} that, if $p^r \neq 2$, then $(1 + a x + b x^2)^{1 /
  p^r} \in \mathbb{Z}_p [[x]]$ if and only if $a, b \in p^{r + 1}
  \mathbb{Z}_p$. On the other hand, as in the previous example, $(1 + a x + b
  x^2)^{1 / 2} \in \mathbb{Z}_2 [[x]]$ if and only if either $a, b \in
  4\mathbb{Z}_2$ or $(a, b) \equiv (2, 1) \pmod{4}$. Let
  $k$ be the denominator of $\lambda$. It follows that $(1 + a x + b
  x^2)^{\lambda} \in \mathbb{Z} [[x]]$ if and only if
  \begin{itemize}
    \item $a, b \in k \operatorname{rad} (k) \mathbb{Z}$, or
    
    \item $k = 2 \kappa$ and $a, b \in \kappa \operatorname{rad} (\kappa) \mathbb{Z}$
    as well as $(a, b) \equiv (2, 1) \pmod{4}$.
  \end{itemize}
  Note that, in the second case, $\kappa$ is necessarily odd.
\end{example}

The remainder of this paper is organized as follows. Before proving
Theorem~\ref{thm:ps:int} in Section~\ref{sec:bin}, we collect some basic
notations and results in Section~\ref{sec:review}. In the final
Section~\ref{sec:dd}, we then review (and slightly extend) the
Dieudonn\'e--Dwork criterion for the purpose of comparison with
Theorem~\ref{thm:ps:int}.

\section{Notations and review}\label{sec:review}

For any formal power series $f (x) \in 1 + x R [[x]]$ over a commutative ring
$R$ containing $\mathbb{Q}$, the power $f (x)^{\lambda}$ can be defined for
any $\lambda \in R$ as a formal power series by the binomial expansion
\begin{equation}
  (1 + a_1 x + a_2 x^2 + \ldots)^{\lambda} = \sum_{n \geq 0}
  \binom{\lambda}{n} (a_1 x + a_2 x^2 + \ldots)^n \in R [[x]], \label{eq:bin}
\end{equation}
where
\begin{equation}
  \binom{\lambda}{n} = \frac{\lambda (\lambda - 1) \cdots (\lambda - n +
  1)}{n!} . \label{eq:bin:coeff}
\end{equation}
We refer to \cite{kauers-paule-ct} and \cite{sambale-fps} for
introductions to formal power series. We note that the binomial expansion
\eqref{eq:bin} for $(1 + x)^{\lambda}$ is equivalent to the definition $(1 +
x)^{\lambda} = \exp (\lambda \log (1 + x))$ used in \cite{sambale-fps} to
which we refer the reader for a detailed discussion of the formal power series
$\exp (x), \log (1 + x) \in \mathbb{Q} [[x]]$ and their properties. In
particular, the expansion \eqref{eq:bin} agrees with products of power series
in the case where $\lambda \in \mathbb{Z}_{\geq 0}$, and the usual power
laws hold. For instance, as pointed out in
\cite[{\textsection}5.1]{kauers-paule-ct}, the multiplication law $f
(x)^{\lambda} f (x)^{\mu} = f (x)^{\lambda + \mu}$ is equivalent to the
Vandermonde convolution identity. Similarly, for any $\lambda, \mu \in R$ and
$f (x) \in 1 + x R [[x]]$, we have $(f (x)^{\lambda})^{\mu} = f (x)^{\lambda
\mu}$, which can be deduced from the corresponding property $\exp
(x)^{\lambda} = \exp (\lambda x)$ of the exponential. Note that this implies
the following observation: if $f (x) = g (x)^{\lambda}$ for $f (x), g (x) \in
1 + x R [[x]]$ and $\lambda \in R^{\times}$, then we necessarily have $g (x) =
f (x)^{1 / \lambda}$.

For our purposes, we note that, if $\lambda \in \mathbb{Z}_p$, then
$\binom{\lambda}{n} \in \mathbb{Z}_p$ for any integer $n \geq 0$ (this is
clear for $\lambda \in \mathbb{Z}_{\geq 0}$ and follows for $\lambda \in
\mathbb{Z}_p$ by $p$-adic continuity and the fact that $\mathbb{Z}_p$ is
closed). Consequently, if $f (x) \in \mathbb{Z}_p [[x]]$ and $\lambda \in
\mathbb{Z}_p$, then $f (x)^{\lambda} \in \mathbb{Z}_p [[x]]$. In particular,
if $\lambda \in \mathbb{Z}_p^{\times}$, then $f (x) \in \mathbb{Z}_p [[x]]$ if
and only if $f (x)^{\lambda} \in \mathbb{Z}_p [[x]]$.

Finally, we recall the following rather well-known result concerning
congruences involving power series (see, for instance,
\cite[Proposition~1.9]{ry-diag13}).

\begin{lemma}
  \label{lem:f:g:pr}For any $f (x), g (x) \in \mathbb{Z}_p [[x]]$ and $r \in
  \mathbb{Z}_{> 0}$,
  \begin{equation*}
    f (x) \equiv g (x) \pmod{p} \quad \Longleftrightarrow \quad f
     (x)^{p^{r - 1}} \equiv g (x)^{p^{r - 1}} \pmod{p^r} .
  \end{equation*}
\end{lemma}

\begin{proof}
  By Fermat's little theorem for power series (which follows from the binomial
  expansion \eqref{eq:bin} combined with the fact that $\binom{p}{n}$ is
  divisible by $p$ except if $n = 0$ or $n = p$), we have
  \begin{equation}
    f (x)^p \equiv f (x^p) \pmod{p}, \label{eq:f:fermat}
  \end{equation}
  and thus $f (x)^{p^{r - 1}} \equiv f (x^{p^{r - 1}})$ modulo $p$, for any $f
  (x) \in \mathbb{Z}_p [[x]]$. Therefore, if $f (x)^{p^{r - 1}} \equiv g
  (x)^{p^{r - 1}}$ modulo $p^r$, then $f (x) \equiv g (x)$ modulo $p$.
  
  For the converse, suppose that $f (x) \equiv g (x)$ modulo $p$. We will
  prove $f (x)^{p^{r - 1}} \equiv g (x)^{p^{r - 1}}$ modulo $p^r$ by induction
  on $r$. For $r = 1$, this is true by assumption. Suppose that $f (x)^{p^{r -
  1}} \equiv g (x)^{p^{r - 1}}$ modulo $p^r$ for some $r \geq 1$. Then $f
  (x)^{p^{r - 1}} = g (x)^{p^{r - 1}} + p^r h (x)$ for some $h (x) \in
  \mathbb{Z}_p [[x]]$. Raising both sides to the $p$th power, we find
  \begin{equation*}
    f (x)^{p^r} = (g (x)^{p^{r - 1}} + p^r h (x))^p = g (x)^{p^r} + \sum_{k =
     1}^p \binom{p}{k} g (x)^{p^{r - 1} (p - k)} (p^r h (x))^k .
  \end{equation*}
  For each $k \in \{ 1, 2, \ldots, p \}$, the summand on the right-hand side
  is divisible by at least $p^{r + 1}$, so that we conclude $f (x)^{p^r}
  \equiv g (x)^{p^r}$ modulo $p^{r + 1}$, as needed for the induction step.
\end{proof}

\section{Proof of Theorem~\ref{thm:ps:int}}\label{sec:bin}

In preparation for proving Theorem~\ref{thm:ps:int}, we observe the following
result, which can also be found, for instance, in \cite[Part VIII, Chap.~3,
No.~140]{polya-szego-2}. (Note that it follows from Theorem~\ref{thm:ps:int}
that, more generally, $(1 + a x)^{1 / p^r} \in \mathbb{Z}_p [[x]]$ if and only
if $a \in p^{r + 1} \mathbb{Z}_p$.)

\begin{lemma}
  \label{lem:bin}For any $r \in \mathbb{Z}_{> 0}$, we have $(1 + p^{r + 1}
  x)^{1 / p^r} \in \mathbb{Z}_p [[x]]$.
\end{lemma}

\begin{proof}
  For any $n \in \mathbb{Q}_p$ (in particular, for any $n \in \mathbb{Q}$),
  denote with $\nu_p (n)$ the $p$-adic valuation of $n$, that is the largest
  $r \in \mathbb{Z}$ such that $n / p^r \in \mathbb{Z}_p$ if $n \neq 0$ (and
  $r = \infty$ if $n = 0$). We recall Legendre's formula which states that
  \begin{equation}
    \nu_p (n!) = \frac{n - s_p (n)}{p - 1}, \label{eq:factorial:val}
  \end{equation}
  where $s_p (n)$ is the sum of the digits of $n$ in base $p$. It further
  follows from \eqref{eq:bin:coeff} that, if $\nu_p (\lambda) < 0$, then
  \begin{equation}
    \nu_p \left(\binom{\lambda}{n} \right) = n \nu_p (\lambda) - \nu_p (n!) .
    \label{eq:binom:val:neg}
  \end{equation}
  In particular, by combining this with Legendre's formula, we have
  \begin{eqnarray*}
    \nu_p \left(\binom{p^{- r}}{n} \right) & = & - r n - \nu_p (n!) = -
    \left(r + \frac{1}{p - 1} \right) n + \frac{s_p (n)}{p - 1}\\
    & \geq & - \left(r + \frac{1}{p - 1} \right) n \geq - (r + 1)
    n.
  \end{eqnarray*}
  It therefore follows that the coefficients of $(1 + a x)^{1 / p^r}$ are
  $p$-integral if $\nu_p (a) \geq r + 1$. In particular, $(1 + p^{r + 1}
  x)^{1 / p^r} \in \mathbb{Z}_p [[x]]$, as claimed.
\end{proof}

We are now in a convenient position to prove Theorem~\ref{thm:ps:int} which is
restated in expanded form below. We note that the equivalence of
\ref{i:ps:int:Z:pf} and \ref{i:ps:int:pr:pf} follows from the result of
Heninger, Rains and Sloane \cite[Theorem~1]{hrs-integrality} that we
indicate in Example~\ref{eg:nthroot:mod}.

\begin{theorem}
  \label{thm:ps:int:pf}Suppose $f (x) = 1 + a_1 x + a_2 x^2 + \ldots \in
  \mathbb{Z}_p [[x]]$ and $r \in \mathbb{Z}_{> 0}$. Then the following are
  equivalent.
  \begin{enumerate}
    \item \label{i:ps:int:Z:pf}$f (x)^{1 / p^r} \in \mathbb{Z}_p [[x]]$
    
    \item \label{i:ps:int:pr:pf}$f (x) \equiv g (x)^{p^r} \pmod{p^{r + 1}}$ for some $g (x) \in \mathbb{Z}_p [[x]]$
    
    \item \label{i:ps:int:apr:pf}$f (x) \equiv (1 + a_{p^r} x + a_{2 p^r} x^2
    + \ldots)^{p^r} \pmod{p^{r + 1}}$
  \end{enumerate}
\end{theorem}

\begin{proof}
  Note that $f (x)^{1 / p^r} \in \mathbb{Z}_p [[x]]$ if and only if there
  exists $g (x) \in \mathbb{Z}_p [[x]]$ such that $g (x)^{p^r} = f (x)$. As
  such, \ref{i:ps:int:Z:pf} clearly implies \ref{i:ps:int:pr:pf}.
  
  Obviously, \ref{i:ps:int:apr:pf} implies \ref{i:ps:int:pr:pf}. Let us show
  that the two conditions are, in fact, equivalent. To that end, suppose that
  \ref{i:ps:int:pr:pf} holds. Write $g (x) = b_0 + b_1 x + b_2 x^2 + \ldots$
  and observe that, by repeated application of Fermat's little theorem
  \eqref{eq:f:fermat},
  \begin{equation*}
    g (x)^{p^r} \equiv g (x^{p^r}) = b_0 + b_1 x^{p^r} + b_2 x^{2 p^r} +
     \ldots \pmod{p} .
  \end{equation*}
  Thus, it follows from $f (x) \equiv g (x)^{p^r} \pmod{p^{r + 1}}$ that $b_m \equiv a_{m p^r} \pmod{p}$. In
  other words,
  \begin{equation*}
    g (x) \equiv 1 + a_{p^r} x + a_{2 p^r} x^2 + \ldots \pmod{p}
     .
  \end{equation*}
  By Lemma~\ref{lem:f:g:pr}, this congruence is equivalent to
  \begin{equation*}
    g (x)^{p^r} \equiv (1 + a_{p^r} x + a_{2 p^r} x^2 + \ldots)^{p^r} \pmod{p^{r + 1}},
  \end{equation*}
  so that condition \ref{i:ps:int:apr:pf} follows from \ref{i:ps:int:pr:pf}.
  
  Finally, suppose that condition \ref{i:ps:int:apr:pf} holds. We need to show
  that \ref{i:ps:int:Z:pf} holds, that is, $f (x)^{1 / p^r} \in \mathbb{Z}_p
  [[x]]$. Write $g (x) = 1 + a_{p^r} x + a_{2 p^r} x^2 + \ldots$, and let
  $\Delta (x) = f (x) - g (x)^{p^r}$. By \ref{i:ps:int:apr:pf}, we have
  $\Delta (x) \in p^{r + 1} x\mathbb{Z}_p [[x]]$. Observe that it follows from
  Lemma~\ref{lem:bin} that $(1 + p^{r + 1} x)^{1 / p^r} \in \mathbb{Z}_p
  [[x]]$. In particular, $(1 + h (x))^{1 / p^r} \in \mathbb{Z}_p [[x]]$
  whenever $h (x) \in p^{r + 1} x\mathbb{Z}_p [[x]]$. We therefore conclude
  that
  \begin{equation*}
    f (x)^{1 / p^r} = (g (x)^{p^r} + \Delta (x))^{1 / p^r} = g (x) \left(1 +
     \frac{\Delta (x)}{g (x)^{p^r}} \right)^{1 / p^r} \in \mathbb{Z}_p [[x]]
  \end{equation*}
  because $g (x)$, and thus $g (x)^{- p^r}$ as well, are in $1 + x\mathbb{Z}_p
  [[x]]$.
\end{proof}

\section{The Dieudonn\'e--Dwork criterion}\label{sec:dd}

For comparison with Theorem~\ref{thm:ps:int}, we discuss in this section a
well-known general criterion due to Dieudonn\'e and Dwork that is recorded
as Theorem~\ref{thm:dd} below. We include a proof of this criterion since we
have not seen condition \ref{i:dd:c} stated in the literature (such as
\cite{dwork-norm}, \cite{lang-cyclotomic-2}, \cite{koblitz-p} and
\cite{robert-p}). On the other hand, Dwork's proof for the equivalence of
\ref{i:dd:Z} and \ref{i:dd:q} readily extends to show that \ref{i:dd:Z} and
\ref{i:dd:c} are equivalent as well. As mentioned in \cite{koblitz-p} and
\cite{robert-p}, the Dieudonn\'e--Dwork criterion can be interpreted as
saying that a power series $f (x)$ has $p$-integral coefficients if and only
if it ``commutes to within mod $p$'' with the $p$th power map. Both
\ref{i:dd:q} and \ref{i:dd:c} are natural ways to make such a statement
precise. (We note that, given \ref{i:dd:Z}, the conditions \ref{i:dd:q} and
\ref{i:dd:c} are clearly equivalent. However, it is not obvious to us that the
conditions \ref{i:dd:q} and \ref{i:dd:c} should be equivalent without using
\ref{i:dd:Z} to pass from one to the other.)

\begin{theorem}[Dieudonn\'e--Dwork]
  \label{thm:dd}Let $f (x) \in 1 + x\mathbb{Q}_p [[x]]$. Then the following
  are equivalent:
  \begin{enumerate}
    \item \label{i:dd:Z}$f (x) \in \mathbb{Z}_p [[x]]$
    
    \item \label{i:dd:q}$\displaystyle \frac{f (x)^p}{f (x^p)} \equiv 1 \pmod{p}$
    
    \item \label{i:dd:c}$f (x)^p \equiv f (x^p) \pmod{p}$
  \end{enumerate}
\end{theorem}

\begin{proof}
  We follow the proof in \cite[Chapter~7.2.3]{robert-p}. First, we show that
  \ref{i:dd:Z} implies both \ref{i:dd:q} and \ref{i:dd:c}. Suppose that $f (x)
  \in \mathbb{Z}_p [[x]]$. Then, by Fermat's little theorem
  \eqref{eq:f:fermat}, we have $f (x)^p \equiv f (x^p) \pmod{p}$ which is equivalent to $f (x)^p / f (x^p) \equiv 1$ since $f
  (x^p) \in 1 +\mathbb{Z}_p [[x]]$ is invertible in $\mathbb{Z}_p [[x]]$.
  
  On the other hand, suppose that either \ref{i:dd:q} or \ref{i:dd:c} holds,
  and write $f (x) = 1 + a_1 x + a_2 x^2 + \ldots \in \mathbb{Q}_p [[x]]$. In
  that case, there exists $g (x) = b_1 x + b_2 x^2 + \ldots \in \mathbb{Z}_p
  [[x]]$ such that
  \begin{equation}
    f (x)^p = f (x^p) + p \varphi (x) g (x), \label{eq:dd:fppf}
  \end{equation}
  where $\varphi (x) = f (x^p)$ in the case of \ref{i:dd:q} and $\varphi (x) =
  1$ in the case of \ref{i:dd:c}. It follows that $a_1 = b_1 \in
  \mathbb{Z}_p$. Suppose that $a_j \in \mathbb{Z}_p$ for all $j < n$. We claim
  that $a_n \in \mathbb{Z}_p$ so that \ref{i:dd:Z} follows by induction. To
  show this claim, we compare the coefficient of $x^n$ on both sides of
  \eqref{eq:dd:fppf}. For the left-hand side, we find
  \begin{eqnarray}
    {}[x^n] f (x)^p & = & [x^n]  (1 + a_1 x + \ldots + a_n x^n)^p \nonumber\\
    & = & p a_n + [x^n]  (1 + a_1 x + \ldots + a_{n - 1} x^{n - 1})^p
    \nonumber\\
    & \equiv & p a_n + [x^n]  (1 + a_1 x^p + \ldots + a_{n - 1} x^{(n - 1)
    p}) \pmod{p} \nonumber\\
    & = & p a_n + a_{n / p} .  \label{eq:dd:lhs}
  \end{eqnarray}
  In the final step, we use the convention that $a_{n / p} = 0$ if $p$ does
  not divide $n$. Note that we were able to use Fermat's little theorem
  \eqref{eq:f:fermat} to reduce $(1 + a_1 x + \ldots + a_{n - 1} x^{n - 1})^p$
  since, by the induction hypothesis, all coefficients are in $\mathbb{Z}_p$.
  On the other hand, for the right-hand side of \eqref{eq:dd:fppf}, the only
  coefficients of $f (x^p)$ and $\varphi (x)$ contributing to the coefficient
  of $x^n$ are in $\mathbb{Z}_p$. Further reducing modulo $p$, only $f (x^p)$
  contributes, whose coefficient of $x^n$ is $a_{n / p}$. Upon comparison with
  \eqref{eq:dd:lhs}, we conclude that $p a_n \in p\mathbb{Z}_p$ or,
  equivalently, $a_n \in \mathbb{Z}_p$, as claimed.
\end{proof}

Theorem~\ref{thm:ps:int} was proved by Dwork \cite{dwork-norm} who credits
Dieudonn\'e \cite{dieudonne-exp} for proving the following additive
version (both consider the special case where $f (x)$ has the form $\sum_{n
\geq 0} a_n x^{p^n}$). Since both results are often referred to as
Dwork's lemma in the literature, we find it fitting to include this additive
version here, highlighting that it is consequence of Theorem~\ref{thm:dd}.

\begin{corollary}[Dieudonn\'e--Dwork, additive version]
  \label{cor:dd:exp}Let $f (x) \in x\mathbb{Q}_p [[x]]$. Then $\exp (f (x))
  \in \mathbb{Z}_p [[x]]$ if and only if $f (x^p) - p f (x) \in p\mathbb{Z}_p
  [[x]]$.
\end{corollary}

\begin{proof}
  It follows from Theorem~\ref{thm:dd} that $\exp (f (x)) \in \mathbb{Z}_p
  [[x]]$ if and only if
  \begin{equation*}
    \exp (f (x^p) - p f (x)) \equiv 1 \pmod{p\mathbb{Z}_p [[x]]}
     .
  \end{equation*}
  The claim therefore follows from
  \begin{equation}
    \exp \left(\sum_{n \geq 1} a_n x^n \right) \in p\mathbb{Z}_p [[x]]
    \quad \Longleftrightarrow \quad \sum_{n \geq 1} a_n x^n \in
    p\mathbb{Z}_p [[x]] . \label{eq:dd:exp:pZ}
  \end{equation}
  The ``$\Longleftarrow$'' part of this equivalence is a consequence of the
  fact that $\exp (p x) \in 1 + p x\mathbb{Z}_p [[x]]$ (see, for instance,
  \cite[Chapter~7.2.3]{robert-p}). On the other hand, suppose that the
  left-hand side of \eqref{eq:dd:exp:pZ} holds but that the right-hand side
  does not. In that case, there exists $N \geq 1$ such that $a_N \not\in
  p\mathbb{Z}_p$. Suppose that $N$ is chosen as small as possible. Since $\exp
  \left(\sum_{n < N} a_n x^n \right)$ as well as its inverse are in $1 + p
  x\mathbb{Z}_p [[x]]$, we conclude that $\exp \left(\sum_{n \geq N} a_n
  x^n \right) \in p\mathbb{Z}_p [[x]]$. The coefficient of $x^N$ in that
  series is $a_N$ so that, in particular, $a_N \in p\mathbb{Z}_p$. This,
  however, is a contradiction.
\end{proof}

There are various directions in which Theorem~\ref{thm:dd} and
Corollary~\ref{cor:dd:exp} can be extended, including to several variables and
extensions of $\mathbb{Q}_p$. We refer the interested reader to
\cite{robert-p} for more information. For recent work on truncated versions
of this integrality criterion, we refer to \cite{km-dwork}. Here, for
comparison with Theorem~\ref{thm:ps:int}, we offer the following slight
extension of the Dieudonn\'e--Dwork criterion as stated in
Theorem~\ref{thm:dd} (which is the case $r = 0$ of the following).

\begin{corollary}
  \label{cor:ddx}Let $f (x) \in 1 + x\mathbb{Q}_p [[x]]$ and $r \in
  \mathbb{Z}_{\geq 0}$. Then the following are equivalent:
  \begin{enumerate}
    \item \label{i:ddx:Z}$f (x)^{1 / p^r} \in \mathbb{Z}_p [[x]]$
    
    \item \label{i:ddx:q}$\displaystyle \frac{f (x)^p}{f (x^p)} \equiv 1 \pmod{p^{r + 1}}$
    
    \item \label{i:ddx:c}$f (x)^p \equiv f (x^p) \pmod{p^{r
    + 1}}$
  \end{enumerate}
\end{corollary}

\begin{proof}
  It follows from Theorem~\ref{thm:dd} that conditions \ref{i:ddx:q} and
  \ref{i:ddx:c} imply $f (x) \in \mathbb{Z}_p [[x]]$. Since condition
  \ref{i:ddx:Z} clearly implies $f (x) \in \mathbb{Z}_p [[x]]$ as well, we may
  assume throughout that $f (x) \in 1 + x\mathbb{Z}_p [[x]]$. This assumption,
  in particular, implies that $1 / f (x) \in \mathbb{Z}_p [[x]]$ so that it
  becomes clear that conditions \ref{i:ddx:q} and \ref{i:ddx:c} are
  equivalent.
  
  Recall from Lemma~\ref{lem:f:g:pr} that, for $p$-integral power series, $f
  (x)^{p^r} \equiv g (x)^{p^r} \pmod{p^{r + 1}}$ is
  equivalent to $f (x) \equiv g (x) \pmod{p}$.
  Consequently, condition \ref{i:ddx:q} is equivalent to
  \begin{equation*}
    \left(\frac{f (x)^p}{f (x^p)} \right)^{1 / p^r} \equiv 1 \pmod{p} .
  \end{equation*}
  By the Dieudonn\'e--Dwork criterion as in Theorem~\ref{thm:dd}, applied
  with $f (x)^{1 / p^r}$ in place of $f (x)$, it follows that this latter
  congruence is equivalent to $f (x)^{1 / p^r} \in \mathbb{Z}_p [[x]]$, which
  is condition \ref{i:ddx:Z}.
\end{proof}

In \cite{robert-p}, the Dieudonn\'e--Dwork criterion is described by
stating that ``the extent to which the operations
\begin{itemize}
  \item first raising $x$ to the power $p$ and then applying $f$,
  
  \item first computing $f (x)$ and then raising to the $p$th power
\end{itemize}
lead to similar results, is a measure of the integrality of the coefficients
of $f (x)$.'' Corollary~\ref{cor:ddx} can be interpreted as a quantifiable
version of this statement. Corollary~\ref{cor:ddx} also provides a second
characterization of the $p$-integrality of power series $f (x)^{1 / p^r}$
which differs from the characterization we offer in Theorem~\ref{thm:ps:int}.
To appreciate this difference, we conclude with the following example.

\begin{example}
  As observed in Example~\ref{eg:pr1}, Theorem~\ref{thm:ps:int} immediately
  implies that, if $a_1, a_2, \ldots, a_d \in \mathbb{Z}_p$ with $d < p^r$,
  then $(1 + a_1 x + a_2 x^2 + \ldots + a_d x^d)^{1 / p^r} \in \mathbb{Z}_p
  [[x]]$ if and only if $a_1, a_2, \ldots, a_d \in p^{r + 1} \mathbb{Z}_p$. On
  the other hand, applying Corollary~\ref{cor:ddx} to this case, we find that
  $(1 + a_1 x + a_2 x^2 + \ldots + a_d x^d)^{1 / p^r} \in \mathbb{Z}_p [[x]]$
  if and only if
  \begin{equation*}
    (1 + a_1 x + a_2 x^2 + \ldots + a_d x^d)^p \equiv 1 + a_1 x^p + a_2 x^{2
     p} + \ldots + a_d x^{p d} \pmod{p^{r + 1}} .
  \end{equation*}
  It then requires additional thought to conclude that this congruence is
  equivalent to $a_1, a_2, \ldots, a_d \in p^{r + 1} \mathbb{Z}_p$.
\end{example}

\subsection*{Acknowledgements}

We thank Tewodros Amdeberhan and Christoph Koutschan for valuable comments on
earlier drafts of this paper. We are also grateful to the referee for helpful
comments that made the paper more focused, as well as for suggesting the
present proof of Corollary~\ref{cor:ddx} that avoids use of
Theorem~\ref{thm:ps:int}.

The second author gratefully acknowledges support through a Collaboration
Grant (\#514645) awarded by the Simons Foundation.

\end{document}